\documentclass[12pt,reqno]{amsart}
 
\usepackage{latexsym, amsmath, amssymb, amscd}

\usepackage{mathrsfs}

\newtheorem{theorem}{Theorem}[section]
\newtheorem{lemma}[theorem]{Lemma}
\newtheorem{proposition}[theorem]{Proposition}
\newtheorem{corollary}[theorem]{Corollary}

\theoremstyle{definition}
\newtheorem{definition}[theorem]{Definition}

\theoremstyle{remark}
\newtheorem{remark}[theorem]{Remark}

\numberwithin{equation}{section}

\newcommand{\ts}{\hspace{.11111em}}
\newcommand{\tts}{\hspace{.05555em}}

\newcommand{\Ncal}{\mathcal{N}}
\newcommand{\Ucal}{\mathcal{U}}

\newcommand{\F}{\mathbb{F}}

\DeclareMathOperator{\G}{\operatorname{\mathcal{G}}}  

\DeclareMathOperator{\vol}{\operatorname{\mathsf{Vol}}\tts}
\DeclareMathOperator{\inj}{\operatorname{\mathsf{Inj}}\tts} 
\DeclareMathOperator{\emb}{\operatorname{\mathsf{Emb}}\tts} 

\begin{document}
 
\baselineskip.525cm
 
\title{Embolic volume and Betti numbers}

\author[L.~Chen]{Lizhi Chen}

\thanks{}

 \address{
\hspace*{0.055in}School of Mathematics and Statistics, Lanzhou University \newline
\hspace*{0.175in} Lanzhou 730000, P.R. China 
}

\email{\hspace*{0.025in} lizhi@ostatemail.okstate.edu}

\thanks{Supported by NSFC Grant No. 11901261 , and the Fundamental Research Funds for the Central Universities of No. lzujbky-2017-26 }



\date{\today}

\begin{abstract}
 We show a new result of relating embolic volume of compact manifolds to Betti numbers.  The result is an improvement to Durumeric's previous work. The proof is based on Gromov's method appeared in systolic geometry. 
\end{abstract}

\maketitle 

\tableofcontents

\section{Introduction}
 
Let $M$ be a compact $n$-dimensional manifold. In this paper we show that the embolic volume of $M$ is related to Betti numbers. This paper is a successive part of \cite{Chen2019}. Relation between the first Betti number and embolic volume is discussed in \cite{Chen2019}.  

For a manifold $M$ endowed with a Riemannian metric $\G$, denote it by $(M, \G)$. Let $\inj(M, \G)$ be the injectivity radius of a Riemannian manifold $(M, \G)$. Embolic volume of the compact $n$-manifold $M$, denoted by $\emb (M)$, is defined to be
\begin{equation*}
 \inf_{\G} \frac{\vol_{\G}(M)}{\inj(M, \G)^n},
\end{equation*}
where the infimum is taken over all Riemannian metrics $\G$ on $M$. Let $\F$ be a coefficient field of homology goups. Main result of this paper is as follows. 
\begin{theorem}\label{theorem_main}
 The embolic volume of a $n$-dimensional compact Riemannian manifold satisfies
 \begin{equation} \label{estimate_main}
  \emb(M) \geqslant C_n \frac{b_k (M; \F)}{\exp{ \left( C_n^{\prime} \sqrt{ \log{b_k (M; \F)} } \right) }} ,
 \end{equation}
 where $C_n$ and $C_n^{\prime}$ are constants only depending on the manifold dimension $n$, and can be explicitly given. 
\end{theorem}

Theorem \ref{theorem_main} is similar to Gromov's theorem of systolic volume in systolic geometry, see \cite[Section 6.4.C.]{Gromov1983} or \cite[Section 4.46.]{Gromov1999}. And in fact the proof of Theorem \ref{estimate_main} is basically from a generalization of detailed form of Gromov's method. Berger named Gromv's method as ``covering trick'' (see Berger \cite[Section 7.2.1.2., Lemma 125.]{Berger2003}), which is explicitly explained in \cite{Chen2019}.

Previously, Durumeric showed the following result. 
\begin{theorem}[Durumeric 1989, see {\cite[Section 5.3.]{Durumeric1989}}]
 For a compact $n$-dimensional manifold $M$, its embolic volume $\emb(M)$ and Betti numbers $b_k (M; \F)$ are related as follows,
 \begin{equation*}
  \emb (M) \geqslant C_{n, k} \ts \ts {b_k (M; \F)}^{\frac{1}{k + 1}} ,
 \end{equation*}
 where $C_{n, k}$ is a constant only depending on $n$ and $k$. 
\end{theorem}
\begin{remark}
The method used to prove our theorem is different from Durumeric's approach.
\end{remark}

In the proof of main theorem, we are going to use Croke's local embolic inequality: if $0 < R \leqslant \frac{1}{2} \inj(M, \G) $, the inequality
\begin{equation}\label{Croke}
 \vol_{\G} ( B(p, R) ) \geqslant \beta_n R^n
\end{equation}
holds for every ball $B(p, R) \subset M$ in a compact Riemannian manifold $(M, \G)$, where $\beta_n$ is a constant only depending on the manifold dimension $n$, and a non-optimal value can be explicitly given, see Croke \cite{Croke1980} or Berger \cite[Section 7.2.4.]{Berger2003}. 
\begin{remark}
 Croke's local inequality (\ref{Croke}) is a key step in our proof. There is a similar local estimate of volume of balls in Gromov's theorem of systolic geometry. 
\end{remark} 

In the proof of Theorem \ref{theorem_main}, we construct a covering space of $M$ with balls $B(p, R)$ satisfying $0 < R < \frac{1}{2} \inj(M, \G)$. A covering space is called a good covering, if every subset of the covering is contractible, and every intersection of finite subsets of the covering is also contractible. Note that when $0 < R < \frac{1}{2} \inj(M, \G)$, the ball $B(p, R) \subset M$ is convex and diffeomorphic to an Euclidean ball, so that contractible. Hence the covering space we constructed is a good covering. Applying a refined version of Gromov's covering argument, we show the estimate in Theorem \ref{theorem_main}. By nerve theorem, nerve of a good covering is homotopy equivalent to the manifold $M$. Denote by $P$ underlying polyhedron of the geometric realization of the nerve. Theorem \ref{theorem_main} is yielded by the following theorem.
\begin{theorem}\label{polyhedron}
 For a given $n$-dimensional compact manifold $M$, there exists a polyhedron $P$, such that the number $t_k$ of $k$-simplices of $P$ satisfies
 \begin{equation}
   t_k \leqslant \frac{2 \vol_{\G}(M)}{\beta_n R_0^n} 5^{n + (n + 1) \sqrt{\log_5{\frac{\vol_{\G}(M)}{\beta_n R_0^n}}}} .
 \end{equation}
\end{theorem}
 
Theorem \ref{polyhedron} implies an estimate of embolic volume. 
\begin{corollary} \label{nerve_2}
 For a compact $n$-dimensional manifold $M$, there exists a polyhedron $P$ homotopy equivalent to $M$, such that the number $t_k$ of $k$-simplices in $P$ satisfies
 \begin{equation}
  t_k \leqslant D_n \emb(M) \exp{ \left( \ts D_n^{\prime} \sqrt{ \log{ \emb(M) } } \ts \right) },
 \end{equation} 
 where $D_n$ and $D_n^{\prime}$ are two positive constants only depending on the manifold $n$, and can be explicitly given. 
\end{corollary}

\vskip 10pt

\section{Good covering and proof of main theorem}

\subsection{Nerve theorem}

Let $X$ be a topological space, and $\Ucal$ be a covering space of $X$.
\begin{definition}
 The nerve of $\Ucal$, denoted by $\Ncal (\Ucal)$, is a simplicial complex, such that
 \begin{enumerate}
  \item there is a $0$-simplex (vertex) $u$ corresponding to each subset $U$ in $\Ucal$;
  \item there is a $k$-simplex spanned by $u_{i_1}, u_{i_2}, \cdots , u_{i_k} $, if  
   \begin{equation*}
    U_{i_1} \cap U_{i_2} \cap \cdots U_{i_k} \neq \emptyset   
   \end{equation*}
   holds for the corresponding subsets $U_{i_1}, U_{i_2}, \cdots , U_{i_k} \in \Ucal$.
 \end{enumerate}
\end{definition}
Denote by $| \Ncal (\Ucal) |$ the underlying polyhedron (or the underlying polyhedron of a geometric realization) of the nerve $\Ncal (\Ucal)$. 
\begin{definition}
 An open cover $\Ucal$ of space $X$ is called a good cover, if every subset $U \in \Ucal$ is contractible, and every finite intersection of subsets $U_{i_1}, U_{i_2}, \cdots , U_{i_k}$ of $\Ucal$ is also contractible.  
\end{definition}
\begin{theorem}[Nerve theorem, see Hatcher {\cite[Corollary 4G.3.]{Hatcher2002}}]
 If $\Ucal$ is a good cover of a paracompact topological space $X$, then underlying polyhedron $P$ of geometric realization of the nerve $\Ncal (\Ucal)$ is homotopy equivalent to $X$. 
\end{theorem}
Moreover, if there exists a partition of unity 
\[ \{ \varphi_U : X \to [0, 1] | U \in \Ucal \} \]
subordinated to the open cover $\Ucal$, then the Alexandroff map (see Dugundji \cite[VIII. \textsection 5.]{Dugundji1966}) $\varphi: X \to | \Ncal (\Ucal) |$ defined by
\[ x \in X \mapsto \sum_{U \in \Ucal} \varphi_{U} (x) \ts u \in | \Ncal (\Ucal) |  \]
is a homotopy equivalence. 

\vskip 10pt

\subsection{Covering of good balls}

Given a Riemmannian manifold $(M, \G)$, we define balls with some growth condition to be good balls. 
\begin{definition}
 Let $R_0$ and $\theta$ be two positive constants. For $0 < R \leqslant R_0$, a ball $B(p, R) \subset M$ with center $p$ and radius $R$ is called a $\theta$-good ball, if the following are satisfied:
 \begin{enumerate}
  \item $ \vol_{\G} (B(p, 5 R)) \leqslant \alpha \vol_{\G} ( B(p, R) ) $ holds for $\alpha = 5^{n + \theta}$,
  \item $ \vol_{\G} ( B(p, 5R^{\prime}) ) > \alpha \vol_{\G} ( B(p, R^{\prime}) ) $ holds for any $R^{\prime} > R$. 
 \end{enumerate}
\end{definition} 
\begin{proposition}
 For any given $R_0 > 0$ and $\theta > 0$, there exists a $\theta$-good ball at any point $p$ in an $n$-dimensional Riemannian manifold. 
\end{proposition}
\begin{proof}
 The proof relies on the formula of volume of balls with small radii in terms of scalar curvature, see \cite{Chen2019}. 
\end{proof}

Let $\{ B(x_j, R_j) \}_{j = 1}^N$ be a maximal system of disjoint good balls. Hence $ \Ucal = \{ B(p_j, 2 R_j) \}_{j = 1}^N$ is a covering space to $M$. The way of constructing covering like this is called ``covering trick'' by Berger, see  Berger \cite[Section 7.2.1.2, Lemma 125.]{Berger2003} or Chen \cite{Chen2019}. We estimate the number of pairwise intersections in $\Ucal$ in the following. Hence the number of simplices in the nerve $\Ncal (\Ucal)$ is obtained. 

\subsection{Proof of main theorem}

\begin{lemma}
 Assume that $B(p_j, 2 R_j)$ and $B(p_{j_i}, 2 R_{j_i})$ are balls in $\Ucal$. If 
 \begin{equation*}
  B(p_j , 2 R_j) \cap B(p_{j_i}, 2 R_{j_i}) \neq \emptyset
 \end{equation*}
 and $R_j \geqslant R_{j_i}$, then we must have
 \begin{equation}\label{5-ball}
  B(p_{j_i}, R_{j_i}) \subset B(p_{j}, 5 R_{j_i}) . 
 \end{equation}
\end{lemma}
\begin{proof}
 The relation (\ref{5-ball}) holds according to the triangle inequality. 
\end{proof}

Now we take $R_0 = \frac{1}{2}\inj(M, \G)$. For a good ball $B(p, R) \subset M$, let 
\begin{equation*}
 5^{-k} R_0 < R \leqslant 5^{-k + 1} R_0
\end{equation*} 
for some integer $k \geqslant 1$. When $R = R_0$, we take $k = 0$. The integer $k$ thus depends on $p$. An upper bound of $R$ can be obtained in terms of the growth condition in the definition of good balls (see \cite{Chen2019} for details),
\begin{equation*}
 k < \frac{\log_5 \frac{\vol_{\G}(M)}{\beta_n R_0^n}}{-n + \log_5 \alpha} .
\end{equation*}  
Let $\log_5 \alpha = n + \sqrt{ \log_5 \frac{\vol_{\G}(M)}{\beta_n R_0^n} }$. That is, we choose $\theta = \sqrt{ \log_5 \frac{\vol_{\G}(M)}{\beta_n R_0^n} }$. The upper bound of $k$ now is equal to $\theta$. 

For the maximal system $\{ B(p_j, R_j) \}_{j = 1}^N$ of disjoint good balls constructed above, we further assume that the following condition holds: 
\begin{enumerate}
 \item $R_1 \geqslant R_2 \geqslant \cdots \geqslant R_N $,
 \item  $B(p_j, 2R_j) \cap B(p_i, 2R_i) = \emptyset $ if $j > i$.
\end{enumerate}

Let the number of pairwise intersections of balls in $\Ucal = \{ B(p_j, 2R_j) \}_{j = 1}^N$ be denoted by $T$. We have the following estimate about $T$,
\begin{align*}
 \vol_{\G}(M) & > \sum_{j = 1}^N \vol_{\G} ( B(p_j, R_j) ) \\
 & = \alpha^{-1} \cdot \sum_{j=1}^N \alpha \vol_{\G} (B(p_j, R_j)) \\
 & \geqslant \alpha^{-1} \cdot \sum_{j = 1}^N \vol_{\G} ( B(p_j, 5 R_j) ) \\
 & \geqslant \alpha^{-1} \cdot \sum_{j = 1}^N \sum_{i = 1}^{T_j} \vol_{\G} ( B(p_j, R_{j_i}) ) \\
 & > \alpha^{-1} \beta_n ( 5^{-k} R_0 )^n \cdot T, 
\end{align*}
so that
\begin{equation*}
 T \leqslant \frac{\vol_{\G}(M)}{\beta_n R_0^n} 5^{n + (n + 1) \sqrt{\log_5{\frac{\vol_{\G}(M)}{\beta_n R_0^n}}}} .
\end{equation*}
That is, 
\begin{equation*}
 T \leqslant C_n  \ts \frac{\vol_{\G}(M)}{\inj(M, \G)^n} \exp{\left( \ts C_n^{\prime} \sqrt{\log{\frac{\vol_{\G}(M)}{\inj (M, \G)^n}}} \ts \right)},
\end{equation*}
where $C_n$ and $C_n^{\prime}$ are two positive constants only depending on $n$, and can be explicitly given,
\begin{equation*}
 C_n = \frac{2^n}{\beta_n} 5^{2n + (n + 1) \sqrt{ \log_5{\frac{2^n}{\beta_n}} } }, \quad C_n^{\prime} = (n + 1) \sqrt{\log{5}} . 
\end{equation*}

\vskip 10pt

The number $T$ of pairwise intersections is equal to the number $t_1$ of $1$-simplices in the nerve $\Ncal(\Ucal)$. Note that $t_0 \leqslant 2t_1$, and $t_i \leqslant t_1$ for $i \geqslant 2$, hence Theorem \ref{polyhedron} and Corollary \ref{nerve_2} holds. Theorem \ref{theorem_main} is yielded by Theorem \ref{polyhedron}.

\end{document}